\documentclass{amsart}

\usepackage[latin1]{inputenc}
\usepackage[varg]{pxfonts}
\usepackage{graphicx}
\usepackage{mathrsfs}
\usepackage[all]{xy}
\usepackage{amssymb}
\usepackage[active]{srcltx} 
\usepackage[active]{srcltx} 
\usepackage[active]{srcltx} 
 \newtheorem{thm}{Theorem}[section]
 
  \newtheorem{fact}[thm]{Fact}
 \newtheorem{lem}[thm]{Lemma}
 \newtheorem{prop}[thm]{Proposition}
 \theoremstyle{definition}
 \newtheorem{defn}[thm]{Definition}
 
 \theoremstyle{remark}
 
 \numberwithin{equation}{section}
\begin{document}

\title[Chain recurrent set]
 {Topological Entropy, Entropy points and shadowing}

\author{S.A. Ahmadi}

\address{Department of Mathematics, University of Sistan and Baluchestan, Zahedan, Iran}

\email{sa.ahmadi@math.usb.ac.ir, sa.ahmdi@gmail.com}
\subjclass[2010]{Primary 54H20 }
\keywords{Topological entropy, Uniform spaces, Topological shadowing, Entropy point}




\begin{abstract}
In this note we study some properties of topological entropy for non-compact non-metrizable spaces. We prove that if a uniformly continuous self-map $f$ of a uniform space has topological shadowing property then the map $f$ has positive uniform entropy, which extends the similar known result for homeomorphisms on compact metric spaces having shadowing property. 

\end{abstract}
\maketitle

The notion of entropy, as a measure of information content, was first introduced in 1948 by Shannon. The roots of this issue can be traced back to statistical mechanics, which is originated in the work of Boltzmann who studied the relation between entropy and probability in physical systems in 1870's. Entropy has also generalized around 1932 to quantum mechanics by von Neumann.
Topological entropy is a non-negative real number that measures the complexity of systems on topological spaces and it is the greatest type of entropy of a system.
A dynamical system is called deterministic if its topological entropy vanishes. One may argue that the future of a deterministic dynamical system can be predicted if its past is known. In a similar way positive entropy maybe related to randomness and chaos.\\
Topological entropy is introduced in 1965 by Adler, Konheim and McAndrew \cite{1}, and subsequently studied by many researchers, see for instance. For a system given by an iterated function, the topological entropy represents the exponential growth rate of the number of distinguishable orbits of the iterates\cite{2}\cite{3}\cite{5}.
\section{Introduction and terminologies}

\subsection*{Uniform Spaces}
A \textit{uniform space} is 
a set with a uniform structure defined on it. A uniform structure $\mathcal{U}$ on a space $X$ is defined by the specification of a system of subsets of the product $X\times X$. The family $\mathcal{U}$  must satisfy the following axioms:
\begin{itemize}
\item[U1)]
for any $E_1,E_2$ the intersection $E_1\cap
E_2$ is also contained in $\mathcal{U}$, and if $E_1\subset E_2$ and $E_1\in\mathcal{U}$, then
$E_2\in\mathcal{U}$),
\item[U2)]
every set $E\in\mathcal{U}$ contains the diagonal $\Delta_X =
\{(x,x)\;|\; x\in X\}$;
\item[U3)]
 if $E\in\mathcal{U}$, then $E^{T} = \{(y,x)\;|\;(x,y)\in E\} \in\mathcal{U}$;
\item[U4)]
for any $E\in\mathcal{U}$ there is a $\widehat{E}\in\mathcal{U}$ such that $\widehat{E}\circ \widehat{E}
\subset E$, where $\widehat{E}\circ \widehat{E} =\{(x,y)\;|\; \textrm{ there is a } z\in
X\textrm { with }(x,z)\in \widehat{E}, (z,y)\in \widehat{E}\}$.
\end{itemize}
The elements of $\mathcal{U}$ are called entourages of the uniformity defined
by $\mathcal{U}$. If $(X,\mathcal{U})$ is a uniform space, then the uniform topology on $X$ is the topology in which a neighborhood base at a point $x\in X$ is formed by the family of sets $E[x]$ where $E$ runs through the entourages of $X$ and $E[x]=\{y\in X|(x,y)\in E\}$ is called the cross section of $E$ at $x\in X$.\\
\subsection*{Topological Entropy}
By a cover $\alpha$ we mean a family of open sets whose union is $X$. A cover $\beta$ is a refinement of another cover $\alpha$, which we write as $\beta\succcurlyeq\alpha$, provided that any element of $\beta$ is contained in an element of $\alpha$. Unlike for partitions, it no longer holds that each element of $\alpha$ is then a union of some elements of $\beta$. A join of two covers $\alpha\vee\beta$ is defined by
$\alpha\vee\beta=\{U\cap V:~U\in\alpha,~V\in\beta\}$. 
Clearly,  $\alpha\vee\beta\succcurlyeq\alpha$ and $\alpha\vee\beta\succcurlyeq\beta$. A sub-cover of a cover U is any subfamily $\beta\\subset\alpha$ which is also a cover. For a cover $\alpha$ we let $N(\alpha)$ denote the minimal cardinality of a sub-cover. A sub-cover of this cardinality will be referred to as \textit{optimal}. Let $f:X\rightarrow X$ be a continuous transformation.,If $\alpha$ is an open cover, then $f^{-1}(\alpha)=\{f^{-1}(U):~U\in\alpha\}$ is also a cover and $N(f^{-1}(\alpha))\leq N(\alpha)$. We will denote
$\alpha\vee f^{-1}(\alpha)\vee\dots\vee f^{-(n-1)}(\alpha)$ by $\bigvee_{i=0}^{n-1}f^{-i}(\alpha)$.\\
\begin{fact}
If $\alpha$ is an open cover of $X$ and $f:X\rightarrow X$ is a continuous map then the limit $\lim_{n\rightarrow\infty}\frac{1}{n}\log N(\bigvee_{i=0}^{n-1}f^{-i}(\alpha))$ exists.
\end{fact}
Then the entropy of $f$ with respect to $\alpha$ is defined to be
$$h(f,\alpha)=\lim_{n\rightarrow\infty}\frac{1}{n}\log N(\bigvee_{i=0}^{n-1}f^{-i}(\alpha))$$
 and the \textit{topological entropy} of $f$ is given by 
$$h_{\textsf{top}}(f)=\sup\{h(f,\alpha):~\alpha \mbox{ is an open cover of }X\}.$$
\subsection*{Metric Entropy}
Let $(X,d)$ be a compact metric space and $f : X \rightarrow X$ be a homeomorphism. Let $d_n(x, y) =\max_{0\leq i\leq n-1}d(f^i(x), f^i(y))$ for all $n\in\mathbb{N}$. Each $d_n$ is a metric on $X$ and the $d_n$'s are all equivalent metrics in the sense that they induce the same topology on $X$.
Fix $\epsilon> 0$ and let $n\in\mathbb{N}$. A set $A$ in $X$ is \textit{$(n,\epsilon)$-spanning} if for every point $x\in X$ there exists a point $y\in A$ such that $d_n(x, y) < \epsilon$. By compactness, there are finite $(n,\epsilon)$-spanning sets. Let $r_n(\epsilon,f)$ be the minimum cardinality of the $(n,\epsilon )$-spanning sets.
A set $A\subset X$ is \textit{$(n,\epsilon)$-separated} if the $d_n$-distance between any two distinct points in $A$ is at least $\epsilon$. Let $s_n(\epsilon,f)$ be the maximum cardinality of $(n,\epsilon)$-separated sets. Then
\[h_{top}(f)=\lim_{\epsilon\rightarrow 0}\lim_{n\rightarrow \infty}\frac{1}{n}\log s_n(\epsilon,f)=\lim_{\epsilon\rightarrow 0}\lim_{n\rightarrow \infty}\frac{1}{n}\log r_n(\epsilon,f)\]
is called the \textit{metric entropy} of $f$ \cite{4}.
\subsection*{Uniform Covering Entropy}
 Let $X,\mathcal{U}$ be a uniform space and let $f$ be a uniformly continuous self-map. For each $U\in\mathcal{U}$, we define 
 $$\textsf{cov}(n,U,f)=N(\bigvee_{i=0}^{n-1}f^{-1}(\mathcal{C}(U))),$$
where $\mathcal{C}(U)=\{U[x]:~x\in X\}$ is called a \textit{uniform cover}.
\begin{fact}
If $U$ is an entourage in $\mathcal{U}$ and $f:X\rightarrow X$ is a continuous map then the limit $\lim_{n\rightarrow\infty}\frac{1}{n}\log \textsf{cov}(n,U,f)$ exists.
\end{fact}
This gives the notion of \textit{uniform covering entropy}:
$$h_{\textsf{uc}(f)}=\sup\{h_{\textsf{uc}}(f,K):~K \in\mathcal{K}(X)\}$$
where $h_{\textsf{uc}}(f,K)=\sup\{\textsf{cov}(U,K,f):~U\in\mathcal{U}\}$.

\subsection*{Uniform Entropy}
In continue we give the definition of topological entropy using separating and spanning sets in uniform space $(X,\mathcal{U})$. Given an entourage $E$ and natural number $n$, a subset $A\subseteq X$ is called an $(n,E,f)$-spanning set provided that
$$X\subseteq \bigcup_{x\in A}\left(  \bigcap_{i=0}^{n-1}F^{-i}(E)\right) [x],$$
Or equivalently if for every $x\in X$ there exists $y\in A$ such that $(f^i(x),f^i(y))\in E$ for all $i=0,1,2,\dots,n-1$. By compactness, there are finite $(n,E,f)$-spanning sets. Let $\textsf{span}(n,E,f)$ be the minimum cardinality of an $(n,E,f)$-spanning set.\\

A subset $A\subseteq X$ is called an $(n,E,f)$-separated set provided that
$$A\times A\cap\Delta^c\subset\bigcup _{i=0}^{n-1}F^{-i}(E^c),$$
Or equivalently if for each pair of distinct points $x$ and $y$ in $A$ there exists $0\leq i\leq n-1$ such that $(f^i(x),f^i(y))\notin E$.Again by compactness any $(n,E,f)$-separated set is finite. Let $\textsf{sep}(n,E,f)$ be the maximum cardinality of an $(n,E,f)$-separated set.
\begin{fact}
Let	$f$	be	a	uniformly	continuous	self-map	on	the	uniform	space	$( X , \mathcal{U}	)$.	If $U$ and $V$ be two entourages in $\mathcal{U}$ with $U\subseteq V$, then
\begin{enumerate}
\item
$\textsf{sep}(n,V,f)\leq \textsf{sep}(n,U,f)$ for all $n\geq 0$
\item
$\textsf{span}(n,V,f)\leq \textsf{span}(n,U,f)$ for all $n\geq 0$
\item
$\textsf{span}(n,V,f)\leq \textsf{sep}(n,U,f)$ for all $n\geq 0$
\end{enumerate}
\end{fact}
Given $U\in\mathcal{U}$, define
\begin{align*}
&\textsf{span}(U,f)=\limsup_{n\rightarrow\infty}1/n\log \textsf{span}(n,U,f);\\
&\textsf{sep}(U,f)=\limsup_{n\rightarrow\infty}1/n\log \textsf{sep}(n,U,f).
\end{align*}
And we define the following quantities for uniformly continuous map $f$:
\begin{align*}
&h_{\textsf{span}}(f)=\sup\{\textsf{span}(U,f):~U\in\mathcal{U}\};\\
&h_{\textsf{sep}}(f)=\sup\{\textsf{sep}(U,f):~U\in\mathcal{U}\}.
\end{align*}

\begin{fact}
Let	$f$	be	a	uniformly	continuous	self-map	on	the	uniform	space	$( X , \mathcal{U}	)$. Then $h_{\textsf{span}}(f)=h_{\textsf{sep}}(f)$.
\end{fact}
This gives the notion of \textit{uniform entropy}:
 $$h_{\textsf{u}}(f)=h_{\textsf{span}}(f)=h_{\textsf{sep}}(f)$$
\begin{fact}
 Let $(X,\mathcal{U})$ be a uniform space. If $f:X\rightarrow X$ is a uniformly continuous map, then $h_{\textsf{uc}}(f)=h_{\textsf{u}}(f)$.
\end{fact}
\begin{fact}\label{f1}
 Let $(X,\mathcal{U})$ be a compact uniform space. If $f:X\rightarrow X$ is a continuous map, then $h_{\textsf{uc}}(f)=h_{\textsf{u}}(f)=h_{\textsf{top}}(f)$.
\end{fact}
\subsection*{Expansive maps:}
To simplify we introduce the following notation. If $f:X\rightarrow X$ is a map of a uniform space $(X,\mathcal{U})$ and $D$ is an entourage of $X$, then we define
$$
\Gamma^+(x,D,f)=\left(\bigcap_{i\in\mathbb{N}}F^{-i}(D)\right)[x];\qquad
\Gamma(x,D,f)=\left(\bigcap_{i\in\mathbb{Z}}F^{-i}(D)\right)[x].
$$

We say that a map $f:X\rightarrow X$ of a uniform space $(X,\mathcal{U})$ is a positively expansive (resp. expansive) if there is an entourage $D$ such that  $\Gamma^+(x,D,f)=\{x\}$  (resp $\Gamma(x,D,f)=\{x\}$ ) for all $x\in X$. We shall call such a $D$ a positive expansivity neighborhood (resp. expansivity neighborhood) of $f$. As one of notions that are weaker than expansivity we capture the notion which is called sensitive dependence on initial conditions, This is defined by the property that if there is an entourage $D$ such that for each $x\in X$ and each entourage $U$, we obtain $U[x]\cap(X\setminus\Gamma(x,D,f))\neq\emptyset$\cite{7}.\\

\section{Main Results}
we say that a map $f:X\rightarrow X$ of a uniform space $(X,\mathcal{U})$ is a contraction if for every entourage $D$ of $X$ there is an entourage $U\subset D$ of $X$ satisfying
$$f(U[x]) \subset U[f(x)],~\forall x\in X$$
\begin{thm}
Let $(X,\mathcal{U})$ be a totally bounded uniform space. If $f:X\rightarrow X$ is a contraction then $h_{\textsf{u}}(f)=h_{\textsf{uc}}(f)=0$
\end{thm}
\begin{proof}
Assume that $D$ be any entourage of $X$. Then there exists an entourage $U\subset D$ such that $f(U[x])\subset U[f(x)]$ for all $x\in X$. Since $X$ is totally bounded there exists a finite set $A$ such that $U[F]=X$. We shall show that for any $n\in\mathbb{N}$, the set $A$ is an $(n,D,f)$-spanning set. For each $x\in X$ there exists a point $y\in A$ such that $(x,y)\in U$. Hence $y\in U[x]$ and we obtain $f(y)\in f(U[x])\subset U[f(x)]$. Hence $(f(x),f(y))\in U\subset D$ and $f^2(x)\in f(U[f(x)])\subset U[f^2(x)]$, so $(f^2(x),f^2(y))\in U\subset D$. Inductively we obtain $\cup_{i=1}^{n}(f^i(x),f^i(y))\in D$. This implies that $A$ is an $(n,D,f)$-spanning set. Thus $h_{\textsf{uc}}(f)=h_{\textsf{u}}(f)=h_{\textsf{span}}(f)=0$
\end{proof}
\begin{defn}
Let $(X,\mathcal{U})$ be a compact uniform space and let $U$ be an entourage of $X$. A finite uniform cover $\mathcal{C}(U)=\{U[x]:~x\in X\}$ is called a \textit{uniform generator} for $f$, provided that for any bi-sequence $\{A_n\}_{i\in\mathbb{Z}}\subset\mathcal{C}(U)$, the intersection $\cap_{n=-\infty}^{\infty}f^{-n}(\textsf{cl}(A_n))$ contains at most one point.
\end{defn}
\begin{lem}\label{lem1}
Let $\alpha$ be an open covering of the compact uniform space $(X,\mathcal{D})$. Then there exists a symmetric entourage $D$ such that each member of $\mathcal{C}(D)$ is contained in some member of $\alpha$.
\end{lem}
\begin{prop}
Let $(X,\mathcal{U})$ be a totally bounded uniform space and let $f:X\rightarrow X$ be a continuous map. Then $f$ is expansive if and only if $f$ has a uniform generator.
\end{prop}
\begin{proof}
Assume that $f$ is expansive with expansivity neighborhood $D$. Choosing an entourage $\widehat{D}$ with $\widehat{D}^6\subset D$, since $X$ is totally bounded, there exists a finite set $F=\{x_1,\dots,x_m\}$ such that $\widehat{D}[F]=X$. Let $\{A_n\}_{n\in\mathbb{Z}}\subset\mathcal{C}(\widehat{D})=\{\widehat{D}[x_i]:~i=1,2,\dots,m\}$ and $x,y\in\cap_{n=-\infty}^{\infty}f^{-n}(\textsf{cl}(A_i))$. Then $f^n(x),f^n(y)\in \textsf{cl}(A_n)$ for all $n\in\mathbb{Z}$. In other hand, for each $n\in\mathbb{Z}$ there exists $x_n$ such that $\textsf{cl}(A_n)\subset\textsf{cl}(\widehat{D}[x_n])\subset\widehat{D}^3[x_n]$. Thus $(f^n(x),f^n(y))\in\widehat{D}^6\subset D$. This implies that $y\in\Gamma(x,D,f)$, that is $y=x$.\\
Conversely suppose that $\mathcal{C}(U)$ is a generator for $f$. Let $D$ be the entourage given by Lemma \ref{lem1} and choose an entourage $\widehat{D}$ such that $\widehat{D}^2\subset D$. We shall show that $\widehat{D}$ is an expansivity neighborhood for $f$. Assume that $y\in\Gamma(x,\widehat{D},f)$, then $f^n(x)\in \widehat{D}[f^n(y)]$. Hence there exists a member $A_n$ in $\alpha$ such that $\widehat{D}[f^n(y)]\subset D[f^n(y)]\subset A_n$. Thus $f^n(y)\in\widehat{D}[f^n(x)]\subset \widehat{D}^2[f^n(y)]\subset D[f^n(y)]\subset A_n$. Therefore
$$x,y\in\bigcap_{n=-\infty}^{\infty}f^{-n}(\textsf{cl}(A_n)).$$
Since $\mathcal{C}(U)$ is a uniform generator $y=x$. This complete the proof.
\end{proof}
A sequentially compact uniform space is totally bounded. Therefore  a sequentially compact, non-compact uniform space is sequentially complete, non-complete (totally bounded). The first ordinal space is an example of one such uniform space \cite{8}.
\begin{thm}\label{t1}
Let $(X,\mathcal{U})$ be a sequentially compact uniform space and let $f:X\rightarrow X$ be a homeomorphism. Let $\alpha$ be a uniform generator for $f$. Then for each entourage $E$ of $X$ there exists $n>0$ so that for each element $A$ of $\bigvee_{i=-n}^{n}f^{-i}(\alpha)$ there is a point $x\in A$ such that $A\subseteq E[x]$.
\end{thm}
\begin{proof}
For contracting a contradiction suppose that there exists an entourage $E$ of $X$ so that for all $j>0$ there exist $A_{j,i}\in\alpha$, $-j\leq i\leq j$ and there exist $B_j\in\bigcap_{i=-j}^{j}f^{-i}(A_{j,i})$ such that for each $x_j\in B_j$ there exists $y_j\in B_j$ with $(y_j,x_j)\notin E$. By sequentially compactness there exists a subsequence $\{j_k\}$ of natural numbers such that $x_{j_k}\rightarrow x$ and $y_{j_k}\rightarrow y$ for some $x$ and $y$ in $X$. Then $x\neq y$. Since $\alpha$ is finite, infinitely many $A_{j_k,0}$ coincide. Then $x_{j_k},y_{j_k}\in A$ for infinitely many $K$ and hence $x,y\in \textsf{cl}(A_0)$. Similarly for each $m$, infinitely many $A_{j_k,m}$ coincide and we obtain $A_m\in\alpha$ with $x,y\in T^{-m}(\textsf{cl}(A_m))$. Thus $x,y\in\bigcap_{-\infty}^{+\infty}f^{-n}(\textsf{cl}(A_n))$ which is a contradiction.
\end{proof}
\begin{thm}
Let $(X,\mathcal{U})$ be a sequentially compact uniform space and let $f:X\rightarrow X$ be a homeomorphism. If $D$ is an expansivity entourage for $f$, then $\textsf{sep}(f,D)=h_{\textsf{u}}(f)$.
\end{thm}
\begin{proof}
Assume that $K\subset X$ is compact. Choose two entourages $U$ and 
$E$ of $X$ such that $U^2\subset E\subset E^3\subset D$, then $\textsf{cl}(E)\subset D$. We shall show that $h_{\textsf{u}}(f,K,U^2)=h_{\textsf{u}}(f,K,E)$. The inequality $h_{\textsf{u}}(f,K,U^2)\geq h_{\textsf{u}}(f,K,E)$ is immediate. Consider $x,y\in X$, by positively expansiveness there exists $i\in\mathbb{N}$ such that $(f^i(x),f^i(y))\notin D$. Then $(f^i(x),f^i(y))\notin \textsf{cl}(E)$. Since $K\times K\setminus U$ is compact and 
$$K\times K\setminus U\subset\bigcup_{i\in\mathbb{N}}F^{-i}(K\times K\setminus \textsf{cl}(E)),$$
there exists $k\in\mathbb{N}$ such that
$$K\times K\setminus U\subset\bigcup_{i=1}^kF^{-i}(K\times K\setminus \textsf{cl}(E)).$$
Suppose that $A$ is an $(n,U,f)$-separated set for $K$, then $f^{-k}(A)\cap K$ is an $(n_2k,E,f)$-separated set for $K$. Indeed $x,y\in f^{-k}(A)$ implies that $f^k(x),f^k(y)\in A$. Hence there exists $0\leq n-1$ such that $(f^{i+k}(x),f^{i+k}(y))\notin U$. Thus there exists $1\leq j\leq k$ such that $(f^{i+k+j}(x),f^{i+j+k}(y))\in X\times X\setminus \textsf{cl}(E)\subset X\times X\setminus E$. Therefore by Lemma ??? $h_{\textsf{u}}(f,K,U^2)\leq h_{\textsf{u}}(f,K,E)$. This implies that $\textsf{sep}(f,D)\geq h_{\textsf{u}}(f)$.

\end{proof}

\begin{thm}
Let $(X,\mathcal{U})$ be a compact uniform space and let $f:X\rightarrow X$ be a homeomorphism. If $\mathcal{C}(U)$ is a uniform generator for $f$, then $h_{\textsf{top}}(f)=h_{\textsf{u}}(f)=h_{\textsf{uc}}(f)=\textsf{cov}(f,U)$.
\end{thm}
\begin{proof}
By the Fact \ref{f1} it is enough to show that $h_{\textsf{uc}}(f)=h_{\textsf{uc}}(f,\alpha)$. Let $V$ be any entourage of $X$ and $D$ be the entourage as Lemma \ref{lem1} for the uniform cover $\mathcal{C}(V)$. By Theorem \ref{t1} there exists $n\in\mathbb{N}$ such that for each $A\in \bigvee_{i=-n}^{n}f^{-i}(\mathcal{C}(U))$ there exists $x\in A$ with $A\subset D[x]$. Thus $\mathcal{C}(V)\preccurlyeq\bigvee_{i=-n}^{n}f^{-i}(\mathcal{C}(U)$. Hence we obtain
\begin{align*}
\textsf{cov}(f,V)
&=\lim_{k\rightarrow\infty}\frac{1}{k}\log N(\bigvee_{j=0}^{k-1}f^{-j}(\bigvee_{i=-n}^{n}f^{-i}(\mathcal{C}(U)))\\
&=\lim_{k\rightarrow\infty}\frac{1}{k}\log N(\bigvee_{j=-n}^{n+k-1}f^{-j}(\mathcal{C}(U))\\
&=\lim_{k\rightarrow\infty}\frac{1}{k}\textsf{cov}(2n+k-1,U,f)\\
&=\textsf{cov}(f,U).
\end{align*}
Since $\textsf{cov}(f,V)\leq \textsf{cov}(f,U)$ for all entourages $V$ of $X$, we conclude that 
$$h_{\textsf{uc}}(f)=\sup\{\textsf{cov}(f,V):~V\in\mathcal{U}\}=\textsf{cov}(f,U)$$.
\end{proof}
\begin{prop}
Let $(X,\mathcal{U})$ be a compact uniform space and let $f:X\rightarrow X$ be a continuous map. Then
\begin{enumerate}
\item
$\textsf{Ent}(X,f)$ is closed;
\item
$\textsf{Ent}(X,f)$ is forward invariant
\item
$\phi(\textsf{Ent}(X,f))=\textsf{Ent}(Y,g)$ 
\end{enumerate}
\end{prop}
\begin{prop}
Let $(X,\mathcal{U})$ be a compact uniform space and let $f:X\rightarrow X$ be a continuous map and $K \subseteq X$ a closed subset. If $h_{\textsf{u}}(f,K)>0$,then $K\cap\textsf{Ent}(X,f)\neq\emptyset$.
\end{prop}
\begin{thm}
Let $(X,\mathcal{U})$ be a compact uniform space and let $f:X\rightarrow X$ be a continuous map. Then
$h_{\textsf{u}}(\textsf{Ent}(X,f),f)=h_{\textsf{u}}(f)$
\end{thm}
we say that a uniform space $(X,\mathcal{U})$ is uniformly locally compact if there is an entourage $U$ of $X$, such that for each $x\in X$, $U[x]$ is compact\cite{RANDTKE1970420}.\\
A uniformity on the set X is separated if the diagonal $\Delta$ coincides with the intersection $\cap$ of the entourages. The uniformity is separated if and only if the uniform topology has the Hausdorff property.
Therefore in a locally compact separated uniform space for every point $x$ in $X$ and every entourage $U$ of $X$ there exists an entourage $C$ such that $C[x]$ is compact and $C[x]\subseteq U[x]$.
\begin{thm}
Let $(X,\mathcal{U})$ be a locally compact separated uniform space and let $f:X\rightarrow X$ be a continuous map with shadowing property. Let $Y\subset X$ be an $f$-invariant closed set. Let $g=f|_Y$ and $G=g\times g$. If there is $x\in \textsf{sen}(g)$ and $(x,x)\in\textsf{int}(\textsf{cl}(\textsf{R}(G)))$, then $x\in \textsf{Ent}(X,f)$ and so $h_{\textsf{u}}(f)>0$.
\end{thm}
\begin{proof}
Since $x\in \textsf{sen}(g)$, there exists $V\in\mathcal{V}$ such that $x\in \textsf{sen}_{V}(g)$. Let  $U$ be any entourage of $X$. Then there exists an entourage $C$ such that $C[x]$ is a compact subset of $X$ and $C[x]\subseteq U[x]$. Choose an entourage $E$ such that $E^3\subset V\cap C$. By shadowing property there exists an entourage $D\subset E$ such that every $D$-pseudo-orbit can be $E$-shadowed by some point in $X$. Since $x\in \textsf{int}(\textsf{cl}(\textsf{R}(G)))$, there exists an entourage $W$ with $W^2\subset D$ such that $W\cap Y\times Y\subset\textsf{int}(\textsf{cl}(\textsf{R}(G)))$. Let $(y,z)\in W\cap\textsf{R}(G)$, then there exists $k\in \mathbb{N}$ such that $G^k(y,z)\notin E^3$. By choosing an entourage $\widehat{D}$ with $\widehat{D}^2\subset D$, since $(y,z)\in\textsf{R}(G)$, there exists $l>k$ such that $(g^l(y),y),(g^l(z),z)\in \widehat{D}$. We shall show that $\textsf{sep}(nl,H,\textsf{cl}(C[x]))\geq 2^n$ for every entourage $H\subset E$ and every $n\in\mathbb{N}$.Let $\xi=\{f^i(y)\}_{i=0}^{l-1}$ and $\eta=\{f^i(z)\}_{i=0}^{l-1}$. Since $(y,z),(y,f^l(y)),(z,f^l(z))\in\widehat{D}$, we conclude that $(z,f^l(y)),(y,f^l(z))\in D$. Hence any sequence $\sigma\{\xi,\eta\}^n$ is an $D$-pseudo-orbit. Therefore there exists a point $w_{\sigma}\in X$ which $E$-shadows the sequence $\sigma$. Assume that the starting point of $\sigma$ is $y$, then $(x,w_{\sigma})\in D\circ E\subset E^2\subset C$. Thus $w_{\sigma}\in C[x]$ (we obtain the same result if the starting point of $\sigma$ is $z$).\\
\textit{Claim: }If $\alpha,~\beta\in\{\xi,\eta\}^n$ are distinct, then there exists $0\leq i\leq nl-1$ such that $f^i(w_{\alpha}),f^i(w_{\beta})\notin E$.\\
\textit{Proof of claim:} Since $\alpha$ and $\beta$ are distinct, they differ in some block, say $(j+1)$-th block. Assume that the $(j+1)$-th block of $\alpha$ is $\xi$ and the $(j+1)$-th block of $\beta$ is $\eta$. Since $(f^k(y),f^k(z))\notin E^3$, there exists $i\in\{n_j,n_j+1,\dots,n_j+l-1\}$ such that $(\alpha_i,\beta_i)\notin E^3$. Suppose by contradiction that $f^i(w_{\alpha}),f^i(w_{\beta})\in E$. Since $f^i(w_{\alpha}),\alpha_i)\notin E$ and $f^i(w_{\beta}),\beta_i)\in E$, we obtain $(\alpha_i,\beta_i)\in E^3$ which is a contradiction. This complete the proof of the claim.\\
Therefore the set $A=\{w_{\sigma}:~\sigma\in\{\xi,\eta\}^n\}$ is an $(nl,E,\textsf{cl}(C[x]))$-separated set. Thus the set $A$ is an $(nl,H,\textsf{cl}(C[x]))$-separated set for all entourages $H\subset E$. Therefore $\textsf{sep}(nl,H,\textsf{cl}(C[x]))\geq \textsf{card}(A)=2^n$. This implies that $$h_{\textsf{top}}(f,\textsf{cl}(C[x]))=h_{\textsf{u}}(f,\textsf{cl}(C[x]))\geq \log 2/l$$. That is $x\in \textsf{Ent}(f,X)$.
\end{proof}




%



\begin{thebibliography}{99}
\setlength{\baselineskip}{.45cm}




\bibitem{1}
\textit{ R. L. Adler and A. G. Konheim and  M. H. McAndrew,} Topological entropy, {Trans. Amer. Math. Soc.}, {\bf114}(1965) 309-319.

\bibitem{2}
\textit{R. Bowen,}  Equilibrium States and the Ergodic Theory of Anosov
Diffeomorphisms, Lecture Notes in Math. 470, Springer, Berlin, 1975.


\bibitem{3}
\textit{R. Bowen and D. Ruelle,}
The ergodic theory of Axiom A flows, {Invent. Math.},  {\bf29}(1975), 181-202.
\bibitem{4}
{D. Burago and Y. Burago and S. Ivanov,}
A course in metric geometry,
AMS, 2001.


\bibitem{5}
\textit{T. Choi and J. Kim,} Decomposition theorem on G -spaces, {Osaka J. Math.}, {\bf46}(2009) 87-104.

\bibitem{6}
 \textit{M. J. Field,} Dynamics and Symmetry, ICP Adv. Texts in Math. 3, Imperial College, London 2007.



\bibitem{7}
\textit{O. V. Kirillova}, Entropy concepts and DNA investigations, {Phys. Letters A}, {\bf274}( 2000) 247-253.

\bibitem{8}
\textit{M. Malziri and  M. R. Molaei}, An extension of the notion of topological entropy,
{Chaos, Solitons and Fractals}, {\bf36}(2008) 370-373.
%



\end{thebibliography}
\end{document}